\documentclass[reqno,12pt]{amsart}

\usepackage[utf8]{inputenc}

\usepackage{enumerate}
\usepackage[margin=1in]{geometry}
\usepackage{ifpdf}
\usepackage{amsmath}
\usepackage{amsfonts}
\usepackage{amssymb}
\usepackage{amsthm}
\usepackage[ocgcolorlinks,hyperfootnotes=false,colorlinks=true,citecolor=blue,linkcolor=blue,urlcolor=blue]{hyperref}
\usepackage{setspace}
\usepackage{amsrefs}
\usepackage{nicefrac}
\usepackage{graphicx}
\usepackage{color}
\usepackage{mathtools}

\usepackage[T1]{fontenc}

\newcommand{\ignore}[1]{}

\newtheorem{thm}{Theorem}[section]

\newtheorem{cor}[thm]{Corollary}

\newtheorem{lemma}[thm]{Lemma}

\newtheorem{question}[thm]{Question}

\newtheorem*{theorem*}{\bf Theorem}
\newtheorem{theoremA}{\bf Theorem}

\newtheorem{corollaryB}{\bf Corollary}

\newtheorem{corollaryC}{\bf Corollary}

\theoremstyle{definition}

\newtheorem{example}[thm]{Example}

\theoremstyle{remark}
\newtheorem{remark}[thm]{Remark}

\author{Achinta Kumar Nandi}
\address{Department of Mathematics, Oklahoma State University,
Stillwater, OK 74078, USA}
\email{acnandi@okstate.edu}

\date{\today}

\title{On perturbations of singular complex analytic curves}


\begin{document}

\begin{abstract}
Suppose $V$ is a singular complex analytic curve inside $\mathbb{C}^{2}$. We investigate when a singular or non-singular complex analytic curve $W$ inside $\mathbb{C}^{2}$ with sufficiently small Hausdorff distance $d_{H}(V, W)$ from $V$ must intersect $V$. We obtain a sufficient condition on $W$ which when satisfied gives an affirmative answer to our question. More precisely, we show the intersection is non-empty for any such $W$ that admits at most one non-normal crossing type discriminant point associated with some proper projection. As an application, we prove a special case of the higher-dimensional analog, and also a holomorphic multifunction analog of a result by Lyubich-Peters \cite{Lyubich-Peters14}.  
\end{abstract}

\maketitle



\section{Introduction} \label{section:intro}

In this note, we study the local topology of a singular complex analytic curve (also referred to as a singular one-dimensional complex analytic subvariety) near its singular points. Let $V$ be a singular one-dimensional complex analytic subvariety of the unit ball $B_2$ in $\mathbb{C}^{2}$, containing the origin as a singular point. We seek to classify all one-dimensional subvarieties $W$ of $B_2$ that must intersect $V$ whenever their distance from $V$ is sufficiently small near a singular point. So we look for conditions on the local complex analytic and geometric properties that control the local topology of these $W$.

An earlier work by Lyubich and Peters \cite{Lyubich-Peters14} solves an important case of the stated problem. They prove that all $W$ that can be locally expressed as a sufficiently small perturbation of a local parametrization of $V$ near a singular point, must intersect $V$. In other words, they show that sufficiently small perturbations of a singular holomorphic disk can not escape the intersection. Lyubich and Peters are led to this discovery in order to classify the non-recurrent Fatou components of moderately dissipative non-elementary polynomial automorphisms $f: \mathbb{C}^{2} \rightarrow \mathbb{C}^{2}$. For an invariant non-recurrent Fatou component $\Omega \subset \mathbb{C}^{2}$ of $f$, the image of a limit map $h: \Omega \rightarrow \overline{\Omega}$ is either a single point or a holomorphic curve. The mentioned result then establishes the smoothness of the limit holomorphic curve.

Another motivation for the main question of this article is due to the previous works of Giraldo-Roeder \cite{Giraldo-Roeder19}, Lebl \cite{Lebl08}, Denkowski \cite{Denkowski16}, and Ebenfelt-Rothschild \cite{Ebenfelt07} on the structure of non-degenerate mappings. Historically the geometry of complex analytic manifolds and varieties under finite mappings is widely studied (see, for example, the books by Rudin \cite{Rudin}, Remmert \cite{Rem57}, Gunning-Rossi \cite{Gunning65}, Griffiths-Harris \cite{GHarris94}, Whitney \cite{Whitney72}, etc). In \cite{Ebenfelt07}, Ebenfelt and Rothschild observe that the image of a germ of a dimension $1$ complex submanifold $X$ of $\mathbb{C}^{n}$ under a finite germ $f: (\mathbb{C}^{n}, 0) \rightarrow (\mathbb{C}^{n}, 0)$ is again a dimension $1$ complex submanifold and the germ of $f$ is transversal to $f(X)$ if and only if the pre-image of $f(X)$ under $f$ is $X$. Their result holds in general for complex submanifolds of higher dimension but they require the germ of the complex submanifold $X$ to be not entirely contained inside the Jacobian variety $\{ \frac{\partial f}{\partial z} = 0\}$ of $f$. Lebl and Denkowski then provided different proofs of the smoothness property, and Giraldo-Roeder settled the codimension $1$ case of Ebenfelt-Rothschild's result without the hypothesis that $X$ not be contained in the Jacobian variety. Jelonek \cite{jelonek23}, in his recent work, has observed that the \emph{Milnor} \emph{numbers} of hypervarieties with isolated singularities can not decrease when pulled back under a finite holomorphic map, and has given a short and swift proof of the main theorem by Giraldo-Roeder. An exciting recent preprint \cite{Nowak} of Nowak establishes the smoothness results of Ebenfelt-Rothschild in all dimensions without the requirement that the germ of the complex submanifold $X$ is not contained inside the Jacobian variety $\{\frac{\partial f}{\partial z} = 0\}$ of $f$. A higher dimensional analog of Lyubich-Peters's result will provide a stronger formulation of Giraldo-Roeder's main theorem.  

The Puiseux parametrization theorem (see \cite{Chirka}, or \cite{Lebl20}) ensures that any one-dimensional variety $V$ is locally expressible as a holomorphic image of a disk. But since a neighborhood $U_p$ of $V$ that captures the local behavior of $V$ near a singular point $p$, is, in general, too big to be an appropriate choice for an arbitrary nearby $W$, the result by Lyubich and Peters does not provide a complete answer to our question.

In this article, we prove the following sufficient condition, partially answering our question:
\begin{theoremA}
Let $V$ be a singular one-dimensional subvariety of some open neighborhood $\Omega_2 \subset \mathbb{C}^{2}$ of the origin and $0 \in V^{\times}$. Moreover suppose the polydisk $D \times \Delta'^{1} \subset \subset \Omega_2$ is a \emph{good} neighborhood of $V$ around $0$. Then there exists $\epsilon > 0$ (depending on the good neighborhood) such that for all one-dimensional subvariety $W$ of $\Omega_2$ with:
\begin{itemize}
    \item $W$ has at most one non-normal crossing type discriminant point associated with the projection onto $D$,
    \item $d_{H} (V, W) < \epsilon$,
\end{itemize}
the intersection $V \cap W$ is non-empty.
\end{theoremA}

The singular set of $V$ is denoted using $V^{\times}$. Following Whitney \cite{Whitney72}, we define the notion of a \emph{good} \emph{neighborhood} in Section $2.1$.

The requirement of admitting at most one non-normal crossing type discriminant point essentially makes $W$ a holomorphic disk inside $D \times \Delta^{'1}$. A version of the higher dimensional analog of the Lyubich-Peters result is proved as an application. To describe the higher dimensional setting, we start with a finite holomorphic mapping $f: \Delta^{n} \rightarrow \mathbb{C}^{n+1}$ with $f(0)= 0$ as a singular point of $f(\Delta^{n})$, and a neighborhood $D^{n} \times \Delta^{'1}$ good for $f(\Delta^n)$ around $0$. Suppose $\pi_{n+1}:~f(\Delta^{n}) \cap (D^{n} \times \Delta^{'1}) \rightarrow D^{n}$ is the proper projection associated with the good neighborhood, $D_f$ is the discriminant variety associated with $\pi_{n+1}$, and $\pi_{n+1}(0)$ is a smooth point of $D_f$. Now $D^{n} \times \Delta^{'1}$ serves as a good neighborhood for any holomorphic image $g(\Delta^n)$ whenever the mapping $g:\Delta^n \rightarrow \mathbb{C}^{n+1}$ is sufficiently close to $f$, and thus the discriminant variety $D_g$ associated with the good neighborhood $D^{n} \times \Delta^{'1}$ can be obtained for each such finite holomorphic $g$. Further, notice that, after possibly a small rotation, a small enough polydisk $D^{(n-1)}_1 \times D_2$ \emph{good} for $D_f$ around $\pi_{n+1}(0)$ can be chosen so that $D_f \cap (D^{(n-1)}_1 \times D_2)$ is a connected complex submanifold. We obtain the following special case of the higher dimensional analog:

\begin{corollaryB}
    Let $f$, $D_f$ and $D^{(n-1)}_1 \times D_2$ are as above. Then there exists $\epsilon > 0$ such that for all finite holomorphic $g: \Delta^{n} \rightarrow \mathbb{C}^{n+1}$ with $D_g \cap (D^{(n-1)}_1 \times D_2)$ being a connected complex submanifold and $\| f -g \| < \epsilon$, the intersection $f(\Delta^{n}) \cap g(\Delta^{n})$ is non-empty.
\end{corollaryB}

A version of Lyubich-Peters's result for holomorphic multifunctions is immediate from Theorem \ref{main}:

\begin{corollaryC}
Suppose $F: D \rightarrow (\mathbb{C}^{2})^{m}_{sym}$ is a holomorphic multifunction, the image $\{F(D) \}$ is an irreducible subvariety of some open subset of $\mathbb{C}^{2}$, and $0 \in \{F(D)\}$ is a singular point of the image of the multifunction. Moreover suppose the polydisk $D_1 \times {\Delta}'^{2} \subset \mathbb{C} \times \mathbb{C}'^{2}$ is a good neighborhood of the multigraph of $F$ about $0$. Then there exists $\epsilon > 0$ such that for every holomorphic multifunction $G: D \rightarrow (\mathbb{C}^{2})^{m}_{sym}$ satisfying:
\begin{itemize}
    \item the multi-graph of $G$ has at most one non-normal crossing type discriminant point associated with the projection onto $D_1$,
    \item $\| F - G\|_{sym} < \epsilon$,
\end{itemize}
we have that $\{F(D)\} \cap \{G(D)\} \neq \emptyset$.
\end{corollaryC}

See Section $2.2$ for the definition of a holomorphic multifunction, and the definition of an image $\{F(D)\}$ of a holomorphic multifunction $F$.

\subsection{Notational Conventions}
We use the terminology and notational conventions of Whitney \cite{Whitney72}. We express a set $H$ as $H^{r} \times H^{'(n-r)}$ to indicate that $H$ is a collection of $n$-tuples, $H^{r}$ is the collection of the first $r$-tuples, and $H^{'(n-r)}$ is the collection of the last $(n-r)$-tuples. We denote the closure of a set $V$ by $\overline{V}$ and the boundary of a set $V$ by $\partial V$. We denote an element $Q$ of $X^{m}/ S_m$ or $X^{m}_{sym}$ by $Q = \langle q_1, \dots, q_m \rangle$. The entries of $Q$ are unordered and they are not necessarily distinct.

\subsection{Acknowledgements}
The author would like to acknowledge Ji{\v r}{\' i} Lebl for introducing the result by Lyubich-Peters, and for his kind and patient guidance throughout. The author is grateful to Roland Roeder for carefully reading the manuscript and suggesting necessary corrections and for many long mathematical discussions. His inputs have greatly helped in finding cleaner proofs of many results in this article. The author is thankful to Anand Patel, Sean Curry, Abdullah Al Helal, and Anthony Kable for helpful discussions on various results of this article.

\section{Background}

All varieties in this article are complex analytic varieties. A detailed treatment of the following objects can be found in the expositions by Whitney \cite{Whitney72} and Chirka \cite{Chirka}.

\subsection{Good Neighborhoods, Proper Projections, and their Discriminant Sets}
We call a subset $V$ of $\mathbb{C}^{n}$ \emph{analytic} near a point $p \in \mathbb{C}^{n}$ or, equivalently, a \emph{variety} near $p$, if there exists an open neighborhood $U_p \subset \mathbb{C}^{n}$, and a set of holomorphic functions $\{ f_j \in \mathcal{O}(U_p): j \in \{1, \dots, n\} \}$ such that $V \cap U_p$ is the common zero set of the holomorphic functions $\{f_1, \dots, f_n\}$. A subset $V$ of $\mathbb{C}^{n}$ is called a \emph{local} \emph{variety} if it is \emph{analytic} near all $p \in V$. 

Suppose $V \subset \mathbb{C}^{n}$ is a local variety, and $p \in V$. An open neighborhood $H = H^{r} \times {H'}^{(n-r)} \subset \mathbb{C}^{r} \times \mathbb{C'}^{(n-r)}$ of $p$ (where $H'^{(n-r)}$ is bounded) is called \emph{good} for $V$ if:
\begin{itemize}
    \item $V \cap (H^{r} \times \partial {H'}^{(n-r)}) = \emptyset,$
    \item $V$ is analytic in $H^{r} \times \overline{{H'}^{(n-r)}}$.
\end{itemize}

Assume $V$ be a pure $m$-dimensional local variety, $p \in V$, and $H = H^{r} \times H'^{(n-r)}$ is a good neighborhood for $V$ around $p$. Suppose $r \geq m$. The projection $\pi_{r}: V \cap H \rightarrow H^{r}$ is a proper and finite holomorphic mapping with its image being an $m$-dimensional subvariety of $H^r$. Sometimes these proper projections are also called good projections.

Any point $q \in H^{r}$ in the image of the projection map $\pi_{r}$ has finitely many pre-images. The set of points in $\pi_{r}(V \cap H)$ where this number of pre-images $\{\pi^{-1}_{r} (q)\}$ is maximal constitutes a Zariski open set inside $\pi_{r}(V \cap H)$. The complement of this set inside $\pi_r (V \cap H)$ is a subvariety of codimension at least one, referred to as the \emph{discriminant} \emph{set} associated with the proper projection $\pi_r$.

If we set $m = r$ in the above discussion, then the image of the proper projection $\pi_r$ is all of $H^{r}$, and the discriminant set is a codimension one subvariety inside $H^{r}$.

A \emph{normal} \emph{crossing} \emph{type} discriminant point $q \in H^r$ is a discriminant point for $V$ such that for a neighborhood $N$ of $q$, the pre-image $\pi_r^{-1}(N)$ is finitely many graphs over $N$ of holomorphic mappings which intersect at some point of $\pi_r^{-1}(q)$, that is, $\pi_r^{-1}(N)$ is composed of finitely many smooth sheets that intersect at some point of $\pi_r^{-1}(q)$. The discriminant points that are not normal crossing type we will simply refer to as \emph{non-normal} \emph{crossing} \emph{type} discriminant points.

\subsection{Holomorphic Multifunctions}

Suppose $X$ is a non-empty set. We denote the $m$-fold symmetric product of $X$ by $X^{m}_{sym}$, defined as $X^{m} / S_m$.

Assume $V \subset \mathbb{C}^{n}$ is a local variety of pure dimension $r$, $p \in V$, and $H = H^{r} \times H^{'(n-r)}$ is a neighborhood that is good for $V$ around $p$. For our purposes, the inverse of the proper projection map $\pi_{r}: V \cap H \rightarrow H^{r}$ motivates the idea of a holomorphic multifunction.

Suppose $q \in H^{r}$, and maximal number of points in a fiber $\pi^{-1}_{r}(q) \subset V \cap H$ is $k$. We can define $\pi^{-1}_{r}: H^{r} \rightarrow (V \cap H)^{k}_{sym}$ as $$ q \mapsto \langle q_1, \dots, q_k \rangle, $$ where $\pi^{-1}_{r}(q) = \{q_1, \dots, q_k\}$, and the $q_j$s are not necessarily distinct. To define the multifunction, we find the holomorphic functions $f_j$ that define $V$ inside $H$ and generate the ideal of $V$ at $p$ and solve the system $f_j(q, \cdot) = 0$, where $q \in H^{r}$. The multifunction $\pi^{-1}_{r} (q) \in (V \cap H)^{k}_{sym}$ is then given by the unordered set $\langle (q, q'_1), \dots, (q, q'_m) \rangle$ counted with multiplicity, where the points $q'_r$ are solutions of the system $f_j(q, \cdot) =0$. As there is no natural way to order these pre-image points, we shall count them as points inside the space of unordered $k$-tuples. 

In general, a multifunction is a map $F: X \rightarrow Y^{m}_{sym}$ from a topological space $X$ to the $m$-fold symmetric product of a topological space $Y$. If $Y$ is Hausdorff then its $m$-fold symmetric product space is also Hausdorff. If $Y$ is a metric space, the distance function $d_Y$ can be used to equip $Y^{m}_{sym}$ with a natural metric space structure. In a similar manner, using the complex structure on $Y$, the symmetric product space is naturally endowed with the structure of a singular complex analytic variety.

A multifunction $F: X \rightarrow Y^{m}_{sym}$ from a complex manifold $X$ to a symmetric product of another complex manifold $Y$ is holomorphic if all possible symmetric polynomials of the elements of $\{F(x)\}$ are holomorphic. The \emph{multigraph} associated with a holomorphic multifunction $F$ is a subset of $X \times Y$ defined as $\{ (x, y) \in X \times Y: \text{y is an entry of the tuple F(x)} \}$. So the inverses of all proper projection mappings of analytic varieties are always holomorphic multifunctions, and the varieties are just the multigraphs of these multifunctions.

The set of points in the image of a holomorphic multifunction $F: D \rightarrow (\mathbb{C}^{2})^{m}_{sym}$ inside $\mathbb{C}^{2}$ is denoted by $\{F(D)\}$ and is defined as $$ \{F(D)\}:= \{f_j (z): F(z) = \langle f_1 (z), \dots, f_m (z) \rangle \in (\mathbb{C}^{2})^{m}_{sym}, z \in D \}.$$

The number of distinct elements in the image set $\{F(x)\}$ of a holomorphic multifunction $F$ is not necessarily constant. The \emph{discriminant} \emph{set} of $F$ is the collection of all points $x$ such that the cardinality of $\{F(x)\}$ is smaller than the maximum.  

\subsection{Hausdorff Distance}

The Hausdorff distance between two non-empty sets $V$ and $W$ is defined as: $d_{H}(V, W):= \inf \{ \epsilon > 0: V \subset W_{\epsilon}, W \subset V_{\epsilon} \}$, where $V_{\epsilon}$ is the union of $\epsilon$-balls centered at points of $V$.

\subsection{The Puiseux Parametrization Theorem}\label{para}

Any irreducible one-dimensional subvariety of an open subset of $\mathbb{C}^{n}$ can be locally written as an image of a holomorphic mapping of the form $(t^{m}, f_2(t), \dots, f_{n}(t))$ from the unit disk in $\mathbb{C}$. This result is well-known as the \emph{Puiseux} \emph{parametrization} \emph{theorem}. We shall briefly discuss the idea of proving this here, and use the technique in the proof of Theorem \ref{main}. A detailed proof can be found in \cite{Chirka} or \cite{Lebl20}.

Suppose $V$ is an irreducible one-dimensional subvariety of the unit ball $B_n$ of $\mathbb{C}^{n}$ with $0 \in V$ as a singular point, and $D \times \Delta^{'(n-1)} \subset \mathbb{C} \times \mathbb{C}^{'(n-1)}$ as good neighborhood about $0$. This means the projection $\pi_1: V \cap (D \times \Delta^{'(n-1)}) \rightarrow D$ is proper. Further, suppose that $0$ is the only discriminant point of $\pi_1$ inside $D$.

To parametrize $V \cap (D \times \Delta^{'(n-1)})$, our strategy is to start off by defining a loop on $\pi^{-1}_1 (\partial D)$. Pick $z_0 \in \partial D$, and choose $w_0 \in \pi^{-1}_1(z_0)$. Starting from $z_0$ and traversing around $\partial D$ once corresponds to creating a curve inside $\pi^{-1}_1( \partial D)$ that joins $w_0$ with some other point in $\pi^{-1}_1(z_0)$. And if we continue to traverse around $\partial D$ with $z_0$ as a base point, we obtain newer curves inside $\pi^{-1}_1(\partial D)$ that join $w_0$ with all other points in $\pi^{-1}_1(z_0)$, and eventually we obtain a loop about $w_0$ due to the irreducibility of $V$. Suppose we need to traverse $m_V$-times around $\partial D$ with $z_0$ as a base point to define a loop about $w_0$.

Now we pre-compose this assignment from $\partial D$ to $\pi^{-1}_1( \partial D)$ with the function $z \mapsto z^{m_V}$ from the unit disk $D$ to itself, so that from $\partial D$ to $\pi^{-1}_1(\partial D)$ we have a well-defined map. Under this mapping, starting from $z_0 \in \partial D$ each arc that creates an angle of $\frac{2 \pi}{m_V}$ at the origin, is being mapped to a curve inside $\pi^{-1}_1 (\partial D)$ that joins only two distinct points of $\pi^{-1}_1 (z_0)$. To define the parametrization on $D \setminus \{0\}$, we analytically continue across each branch and we define it on $\{0\}$ using the \emph{Riemann} \emph{Extension} \emph{Theorem}. Thus we obtain a parametrization $f: D \rightarrow V \cap (D \times \Delta^{'(n-1)})$ of the form $z \mapsto (z^{m_V}, f_2(z), \dots, f_{n}(z))$.

\section{A Sufficient Condition}

In this section, we shall prove Theorem \ref{main} which is the central result of this article. Our strategy is to simultaneously parametrize $V$ and $W$ using Puiseux's technique, and then use the Lyubich-Peters result to prove that $V \cap W \neq \emptyset$. The subtlety in the proof originates mainly from two factors concerning the one-dimensional variety $W$. 
\begin{itemize}
    \item Given any one-dimensional complex analytic variety $W$ with sufficiently small Hausdorff distance $d_{H}(V, W)$ from $V$, the maximal number of distinct branches of $W$ can be \emph{arbitrarily} \emph{large} when the branches are considered locally away from the discriminant set.
    \item To be able to apply the Puiseux parametrization technique, we require that the discriminant set of $W$ be within a small neighborhood of the discriminant set of $V$.
\end{itemize}

In what follows, we make a few observations to overcome these two issues.

\subsection{Separation of branches away from the discriminant set}

Let $V$ be an $r$-dimensional variety and let $\pi^{r}$ be the projection from $V \cap (H^{r} \times \Delta^{'(n-r)})$ onto $H^{r} \subset \mathbb{C}^{r}$. The inverse $(\pi^{r})^{-1}$ can be thought of as a holomorphic multifunction. We show that, away from the discriminant set, locally the branches of $V$ are some definite distance apart. 

\begin{lemma}\label{sep}
 Suppose $V$ is an $r$-dimensional variety, $p \in V$, $H^{n} = H^{r} \times \Delta^{'(n-r)}$ is a neighborhood of $p$ that is good for $V$, and $D_V \subset H^{r}$ is the associated discriminant set. Further suppose $\Delta^{r} \subset H^{r}$ is a pre-compact neighborhood such that $p$ is in $\Delta^{r} \times \Delta^{'(n-r)}$. Then for each open neighborhood $D_1 \subset H^{r}$ of $D_V$, there exists $\epsilon > 0$ such that for all $q \in \Delta^{r} \setminus D_1$ the Hausdorff distance between the components of $(\pi_r)^{-1} ( \overline{B(q, \delta)})$ is greater than~$\epsilon$ whenever $\delta > 0$ is sufficiently small.
\end{lemma}

 \begin{proof}

Let the maximum number points in the fiber $\{(\pi^{r})^{-1}(q)\}$ of $q \in H^{r}$ be $m$. The inverse of the proper projection,  $( \pi^{r} )^{-1} : H^{r} \longrightarrow (V \cap H^{n})^{m}_{sym}$ defined as $$q \mapsto \langle \alpha_1,\dots, \alpha_m \rangle$$ is a holomorphic multifunction, where $V \cap ( \{q \} \times \Delta^{'n-r} ) = \{ \alpha_i \}^{m}_{i=1}.$ We shall denote this multifunction by $F$ from now on.

For any open neighborhood $D_{1} \subset H^{r}$ of the discriminant set $D_V$, we have $$H^{r} \setminus D_{1} \subset \{ q \in H^{r}: | \{ (\pi^{r})^{-1}(q) \}| = m \} .$$ Let $\Delta^r$ be a smaller polydisk with $\overline{\Delta^r} \subset H^r$ as mentioned in the statement. For each $q~\in~\overline{\Delta^{r}}~\setminus~D_1$ the subset of $V \cap H^n$ corresponding to the image $F(q)=\langle f_1(q), \dots, f_m(q) \rangle$ is given by $ \{F(q) \} = \{ f_i(q) \}^{m}_{i=1}$. Define $\phi: \overline{\Delta^{r}} \setminus D_1 \longrightarrow \mathbb{R}$ to be $$ \phi(q) = {\rm inf}_{i \neq j} \{ d( f_i(q), f_j(q) ) : i,j \in \{ 1,..., m \} \}.$$ Since for each $q \in \overline{\Delta^{r}} \setminus D_1$ there are $m$-distinct points in the image $F(q)$ of $F$, we must have $\phi(q) > 0$. Note that $\phi$ is continuous because it is a composition of a holomorphic multifunction and a continuous symmetric function.

As $\phi(q) > 0$ on a compact domain $\overline{\Delta^{r}} \setminus D_1$, we obtain $${\rm inf}~  \{ \phi(q): q \in \overline{\Delta^{r}} \setminus D_1 \} > 0.$$ Choose $\delta > 0$ such that ${\rm inf}~  \{ \phi(q): q \in \overline{\Delta^{r}} \setminus D_1 \} > \delta.$

Now we shall obtain an $\epsilon > 0$ such that locally the Hausdorff distance between the branches above $\overline{\Delta^{r}} \setminus D_1$ is greater than this $\epsilon$.

Pick $q \in \overline{\Delta^{r}} \setminus D_1 $. As $\overline{\Delta^{r}} \setminus D_1$ is compact, there exists $\delta_1 > 0$ such that  $\{ (\pi^r)^{-1} (\overline{B( q, \delta_1)} ) \} $ is a disjoint union of compact connected components, for each $q \in \overline{\Delta^{r}} \setminus D_1$. Hence $$ ( \overline{B( q, \delta_1)} \times \Delta^{'n-r} ) \cap V = \{ F (\overline{B( q, \delta_1)} ) \} = \cup_j B_j$$ is a disjoint union of compact connected portions $B_j$ of the branches of $V \cap H^{n}$ above $\overline{B(q, \delta_1)}$. As $F$ is continuous and $\overline{\Delta^{r}} \setminus D_1$ is compact, there exists $ \delta_1 > 0$ small enough (smaller than our previous choice of $\delta_1$) such that $\forall \alpha, \beta \in \overline{B( q, \delta_1)}$ the distance $d( f_i(\alpha), f_i(\beta) ) < \frac{\delta}{2}$, where $f_i(\alpha), f_i(\beta)$ are situated in the branch $B_i$. 

Thus for different branches $B_1, B_2$ of $\{ F (\overline{B( q, \delta_1)} ) \}$, we have $$ d(f_1(\alpha), f_2(\beta) ) \geq d(f_1(\beta), f_2(\beta)) - d(f_1(\alpha), f_1(\beta)), $$ for each $f_i(\alpha), f_i(\beta) \in B_i.$ Using the continuity of the function $\phi$ we have $$ d(f_1(\beta), f_2( \beta) ) > \delta.$$ Combining this with $d( f_i(\alpha), f_i(\beta) ) < \frac{\delta}{2}$, we obtain $$d( f_1(\alpha), f_2(\beta) ) > \frac{\delta}{2},$$ for all $\alpha, \beta \in \overline{B(q, \delta_1)}$. Thus $d(B_1, B_2) = \emph{inf}~ \{ d(x,y):x \in B_1, y \in B_2 \} \geq \frac{\delta}{2}$. Choose $\epsilon > 0$ such that $3 \epsilon < \frac{\delta}{2}$. So $d( B_1, B_2) > 3 \epsilon$ and as the $B_i$ are compact, it follows that the Hausdorff distance between $B_1$ and $B_2$ $$d_H (B_1, B_2) > \epsilon.$$ This proves the result.
\end{proof}

\subsection{Observations concerning the discriminant set}

In the following, we note that, if $F$ is a holomorphic multifunction and if $G$ is a small perturbation of $F$ then the Hausdorff distance of the corresponding discriminant sets of $F$ and $G$ is small.

\begin{lemma}\label{discr0}
Let $F: \Delta^{r} \rightarrow (\mathbb{C}^{n})^{m}_{sym}$ be a holomorphic multifunction, and $D_{F}$ be its discriminant set. Then for each $\epsilon > 0$ there exists $\delta > 0$ such that for any holomorphic multifunction $G: \Delta^{r} \rightarrow (\mathbb{C}^{n})^{m}_{sym}$ satisfying $d_{sym} (F,G) < \delta$, we have $d_{H} (D_{F},D_{G}) < \epsilon$. 
\end{lemma}

\begin{proof}
 
Define $H_{F}:= \{ z \in \Delta^{r}: |\{ F(z)\}| = m \}$. Without loss of generality, we assume the discriminant set of $F$ to be codimension at least $1$. This forces $H_F$ to be non-empty. Similarly, we can define $H_{G}$ for the holomorphic multifunction $G$ to be the set of points where $\{G(z)\}$ has the maximal $m$ distinct elements. Note that $D_{F} = \Delta^{r} \setminus H_{F}$, and $D_{G} = \Delta^{r} \setminus H_{G}$.

Pick $\epsilon > 0$, and suppose $\Tilde{\Delta}^{r} \subset \Delta^{r}$ is a pre-compact polydisk about $0 \in \mathbb{C}^{r}$ such that the distance of $\partial \Delta^{r}$ is more than $\epsilon$ from $\partial \Tilde{\Delta}^{r}$. Now consider the $\epsilon$-neighborhood of $D_{F}$: $$B(D_{F}, \epsilon):= \{ z \in \Tilde{\Delta}^{r}: \| z - q \| < \epsilon,\text{for some q} \in D_{F} \}.$$

Then $\Tilde{\Delta}^{r} \setminus B(D_{F}, \epsilon) \subset H_{F} \cap \Tilde{\Delta}^{r}$. By the Lemma \ref{sep}, there exists $\epsilon_{1} > 0$ such that the local minimal distance of each branch away from the $\epsilon$-neighborhood of the discriminant set $D_{F}$ is greater than $\epsilon_{1}$.

Choose $\delta$ such that $0 < \delta < \frac{\epsilon_1}{3}$. Suppose $G$ is a holomorphic multifunction satisfying $d_{sym}(F,G) < \delta$. Now by Lemma~\ref{sep}, each branch of $G$ must be inside a $\delta$-neighborhood of the branches of $F$ away from the discriminant set $D_{F}$ of $F$. This means $D_{G}$ must be inside $B(D_{F}, \epsilon)$, otherwise the above will be violated. Thus for all $d_{sym} (F,G) < \delta$, the Hausdorff distance of the discriminant varieties $D_F, D_{G}$ is $d_{H} (D_{F}, D_{G}) < \epsilon.$
\end{proof}

We can re-state the above result in terms of analytic subvarieties as follows:

\begin{lemma}\label{discr1}
Suppose $V$ is an $r$-dimensional complex analytic subvariety of some open neighborhood $\Omega_{n} \subset \mathbb{C}^{n}$, and $D_{V}$ be the discriminant set of $V$ for the projection associated with a \emph{good} neighborhood $D^{n-r} \times \Delta^r \subset \subset \Omega_n$. Then for each $\epsilon > 0$ there exists $\delta > 0$ such that whenever $W$ is a complex analytic subvariety of $\Omega_{n}$ satisfying:
\begin{itemize}
    \item $W$ has the same maximal number of distinct branches as $V$,
    \item $d_{H} (V,W) < \delta$,
\end{itemize}
we have $d_{H}(D_{V}, D_{W}) < \epsilon$, where $D_W$ is the discriminant set of $W$.
\end{lemma}

Note that the \emph{branches} of $V$ here refer to the covering sheets $\pi^{-1}_V(\Delta^{r} \setminus D_V)$ of the proper projection mapping $\pi_V$ from the good neighborhood. The requirement of $W$ admitting the same maximal number of distinct branches as $V$ enables us to express them as holomorphic multifunctions mapping to $(\mathbb{C}^{n})^{m}_{sym}$ for the same~$m$, therefore the Lemma \ref{discr1} immediately follows from Lemma \ref{discr0}.

\subsection{On the maximal number of distinct branches}

For any irreducible $V$, we show in the following, the maximal number of distinct branches of any irreducible $W$ with sufficiently small Hausdorff distance $d_{H}(V, W)$ must be an integer multiple of the maximal number of branches of $V$.

\begin{lemma}\label{no}
    Suppose $V$ is a singular irreducible $r$-dimensional subvariety of some open neighborhood $\Omega_{n} \subset \mathbb{C}^{n}$ of the origin with $0 \in V^{\times}$, and the polydisk $D^{r} \times \Delta^{'(n-r)}$ is a good neighborhood for $V$ around $0$. Then there exists an $\epsilon > 0$ such that for all $r$-dimensional irreducible subvariety $W$ of $\Omega_{n}$ with $d_{H} (V, W) < \epsilon$, the maximal number of distinct branches of $W$ above $D^{r}$ is an integer multiple of the maximal number of distinct branches of $V$.
\end{lemma}

\begin{proof}
    Because $D^{r} \times \Delta^{'(n-r)}$ is a good neighborhood, the map $\pi_{V}: V \cap (D^{r} \times \Delta^{'(n-r)}) \rightarrow D^{r}$ is a proper projection. Note that away from the discriminant set $D_{\pi_V}$ the projection map $\pi_V$ is a $k$-to-$1$ covering. The number $k$ is the maximal number of distinct branches of $V \cap (D^{r} \times \Delta^{'(n-r)})$.

    For any irreducible $r$-dimensional subvariety $W$ of $\Omega_{n}$ with sufficiently small Hausdorff distance $d_H(V,W)$, the neighborhood $D^{r} \times \Delta^{'(n-r)}$ becomes good for $W$ as well. Thus the projection $\pi_W: W \cap (D^{r} \times \Delta^{'(n-r)}) \rightarrow D^{r}$ is proper, and away from the discriminant set $D_{\pi_W}$ the map is an $m$-to-$1$ covering, and the number $m$ is the maximal number of distinct branches of $W \cap (D^{r} \times \Delta^{'(n-r)})$.

    Suppose $p \in D^{r}$ is a point not contained in either of the discriminant sets $D_{\pi_V}$ or $D_{\pi_W}$, and $U \subset D^{r}$ is some ball around $p$ that does not meet either of the discriminant sets $D_{\pi_V}$ or $D_{\pi_W}$. Then both $\pi_V$ and $\pi_W$ are $k$-to-$1$ and $m$-to-$1$ covering from $\pi^{-1}_V(U)$ and $\pi^{-1}_W(U)$ onto $U$ respectively.  

    Using the Lemma \ref{sep} above, for sufficiently small $d_H (V,W)$, each branch of $\pi^{-1}_W(U)$ must lie in an $\epsilon$-neighborhood of exactly one branch of $\pi^{-1}_V(U)$.

    Now the projection $\pi_W$ can be factored into $\pi_W = \pi_V \circ f$, where $f: \pi^{-1}_W(U) \rightarrow \pi^{-1}_V(U)$ is another projection defined as follows. For any $q \in \pi^{-1}_W(U)$ there is a unique $p \in \pi^{-1}_V( \pi_W (q))$ with $\| p - q \| < \epsilon$. Let $f(q) = p$.

    Now since $\pi_W$ and $\pi_V$ are covering maps above $U$, and $U$ is locally path connected, $f$ must also be a finite degree covering from $\pi^{-1}_W(U)$ onto $\pi^{-1}_V(U)$ (see \cite{Hatcher}, page 80, problem 16). So if $f$ is a $\ell$-to-$1$ covering, we obtain $m = k \ell $. Thus the maximal number of distinct branches of $W$ must be an integer multiple of the maximal number of distinct branches of $V$.
\end{proof}

Suppose $V$ is an irreducible, singular, one-dimensional subvariety of some domain $\Omega_2 \subset \mathbb{C}^2$. If we additionally require $W$ to be a one-dimensional subvariety of $\Omega_2$ \emph{admitting} \emph{at} \emph{most} \emph{one} \emph{non-normal} \emph{crossing} \emph{type} \emph{discriminant} \emph{point}, we can improve Lemma \ref{discr1} as follows. We show that, once these conditions are imposed, the requirement of $W$ having the same maximal number of branches as $V$ can be removed.

\begin{lemma}\label{onediscr}
    Suppose $V$ is a singular one-dimensional irreducible complex analytic subvariety of some open neighborhood $\Omega_{2} \subset \mathbb{C}^{2}$ of the origin, $0 \in V^{\times}$, and $D \times \Delta^{'1} \subset \subset \Omega_2$ is good for $V$ around $0$. Then for each $\epsilon > 0$, there exists $\delta > 0$ such that, every one-dimensional irreducible subvariety $W$ of $\Omega_{2}$ with at most one non-normal crossing type discriminant point and with $d_{H}(V, W) < \delta$, has that the non-normal crossing type discriminant point of $W$ must lie within a $\epsilon$-ball around the discriminant point of $V$.
    
\end{lemma}

\begin{proof} Let $\pi_V$ be the proper projection from $V \cap (D \times \Delta^{'1})$ associated with the good neighborhood $D \times \Delta^{'1}$ along $\Delta^{'1}$. Without loss of generality, we assume that $\pi_V(0)$ is the only discriminant point associated with $\pi_V$, and this is a non-normal crossing type discriminant point.

 Choose $\epsilon > 0$ sufficiently small. We \emph{claim} that there exists $\delta > 0$ such that if $d_{H}(V, W) < \delta$ then $W$ must admit a non-normal crossing type discriminant point inside $B(0, \epsilon) \subset D$.

 Consider the projection $\pi_V: \pi^{-1}_V(\partial B(0, \epsilon)) \rightarrow \partial B(0, \epsilon)$. Using the non-trivial monodromy action of the fundamental group $\pi_1 (\partial B(0, \epsilon), x_0)$ on the set $\pi^{-1}_V( x_0)$, we can show that the non-normal crossing type discriminant point of $W$ must lie within $B(0, \epsilon)$. Indeed, if $\pi_W$ only contained normal crossing discriminant points inside $B(0, \epsilon)$, any irreducible component $\pi^{-1}_W( \partial B(0, \epsilon))$ will only have trivial monodromy. But Lemma \ref{sep} forces an irreducible component of $\pi^{-1}_W( \partial B(0, \epsilon))$ to be always sufficiently close to $\pi^{-1}_V( \partial B(0, \epsilon))$, and thus to have non-trivial monodromy. Thus the non-normal crossing type discriminant point of $W$ must lie inside $B(0, \epsilon)$. 
\end{proof}

\subsection{Proof of the Main Result}\label{epsilon}

Before diving into the proof of Theorem \ref{main}, we shall briefly explain a key observation from the proof of the result by Lyubich and Peters \cite{Lyubich-Peters14}. The $\epsilon > 0$ in the Lyubich-Peters theorem does not depend on the parametrization function $f: D \rightarrow \mathbb{C}^{2}$ but only on its image. This is apparent from all three proofs they provide in their paper. Indeed, in their main proof of the result (Proposition $12$, Section $5$ of \cite{Lyubich-Peters14}), the first step is to apply suitable biholomorphic changes of coordinates in domain and range to obtain a convenient equation for the image variety $f(D)$. And the rest of the argument is founded upon this description of the image variety $f(D)$. Thus if $\Tilde{f}: D \rightarrow \mathbb{C}^{2}$ is another holomorphic mapping with $f(D) = \Tilde{f}(D)$, then the same convenient equation for the image variety $f(D)$ describes $\Tilde{f}(D)$, and the rest of the argument can be adapted without any modification and with the same $\epsilon$.

Note that this is precisely what makes Theorem \ref{main} possible. To apply the Lyubich-Peters result on a subvariety $W$ that has possibly $k$-times as many branches as $V$, we need to obtain close enough parametrizations of both $V$ and $W$. So for each different $W$, we need to find a suitable but perhaps many-to-one parametrization of $V$. But since the $\epsilon > 0$ only depends on the image $V$ of these parametrizations, we can expect a result like Theorem \ref{main} to be true. 

Now combining the previous observations we shall prove the main result of this article. Let us restate it for the reader's convenience:

\begin{theoremA}\label{main}

Let $V$ be a singular one-dimensional subvariety of some open neighborhood $\Omega_2 \subset \mathbb{C}^{2}$ of the origin and $0 \in V^{\times}$. Moreover suppose the polydisk $D \times \Delta'^{1} \subset \subset \Omega_2$ is a \emph{good} neighborhood of $V$ around $0$. Then there exists $\epsilon > 0$ (depending on the good neighborhood) such that for all one-dimensional subvariety $W$ of $\Omega_2$ with:
\begin{itemize}
    \item $W$ has at most one non-normal crossing type discriminant point associated with the projection onto $D$,
    \item $d_{H} (V, W) < \epsilon$,
\end{itemize}
the intersection $V \cap W$ is non-empty.

\end{theoremA}

\begin{proof}
    Without loss of generality, suppose $V$ is irreducible. 
    For a subvariety $W$ of $\Omega_2$ with a sufficiently small $d_H(V, W) > 0$, the neighborhood $D \times \Delta^{'1}$ is good for $W$ as well. We denote the proper projections from $V \cap (D \times \Delta^{'1})$ and $W \cap (D \times \Delta^{'1})$ onto $D$ by $\pi_V$ and $\pi_W$ respectively. Suppose $D$ is the unit disk.

    Choose a smaller concentric disk $D_1 \subset D$. Using the Puiseux parametrization technique described in Section \ref{para} (or exercise $6.7.5$, page 146 of \cite{Lebl20}), the set $\pi^{-1}_V(D_1)$ can be written as the image of $\psi_V: D_1 \rightarrow \mathbb{C}^{2}$ defined as $z \mapsto (z^{\ell}, g(z))$, where $\psi_V$ is one-one and onto and $g$ is a holomorphic function. As explained at the beginning of this subsection, by the Lyubich-Peters result there exists an $\epsilon_1 > 0$ depending only on the singular subvariety $\psi_V(D_1)$. We shall use this $\epsilon_1 > 0$ to obtain the $\epsilon$ stated in this theorem.

    Choose $\delta > 0$. Using Lemma \ref{onediscr}, there exists $\epsilon_2 > 0$ such that for all subvariety $W$ of $\Omega_2$ with at most one non-normal crossing type discriminant point $\alpha$ satisfying $d_H(V, W) < \epsilon_2$, the non-normal crossing type discriminant point of $W$ must lie within a $\delta$-ball of the discriminant point $0 \in D$ of $\pi_V$. Using Lemma \ref{onediscr}, this $\delta> 0$ can be chosen independently of the subvariety $W$. 
        
    We seek to parametrize $\pi^{-1}_V(D_1)$ and $\pi^{-1}_W(D_1)$ simultaneously using the Puiseux technique. Suppose $\pi^{-1}_V(D_1)$ has $\ell$-distinct branches. By Lemma \ref{no}, $\pi^{-1}_W(D_1)$ must contain an irreducible component $W_1$ that has $m_W \ell$-distinct branches. Following the Puiseux technique, we need to traverse $m_W \ell$-times around $\partial D_1$ to obtain a loop in $\pi^{-1}_W(\partial D_1) \cap W_1$. In order to parametrize $\pi^{-1}_V(D_1)$ and $W_1 \subset \pi^{-1}_W(D_1)$ simultaneously, we shall pre-compose $\psi_V$ with $z^{m_W}$ to obtain a parametrization of $\pi^{-1}_V(D_1)$. The new mapping $\Tilde{\psi}_{V}: D_1 \rightarrow \psi_V(D_1)$ is thus given by $z \mapsto (z^{\ell m_W}, g(z^{m_W}))$.

    To parametrize $W_1 \subset \pi^{-1}_W(D_1)$, we use the following procedure. Suppose $\phi_{\alpha}$ is an automorphism of the disk $D_1$ that maps $0 \in D_1$ to the non-normal crossing type discriminant point $\alpha \in D_1$. By $\Phi_{\alpha}: D_1 \times \Delta^{'1} \rightarrow D_1 \times \Delta^{'1}$ we denote the automorphism $\Phi_\alpha (z, w) = (\phi_{\alpha} (z), w)$. Note that the complex analytic curve $\Phi^{-1}_{\alpha}(W_1) = \Tilde{W}$ has a non-normal crossing type discriminant point at $0 \in D_1$, and $d_H(V, \Tilde{W})$ is sufficiently small. Using the Puiseux technique, we simultaneously parametrize $\Tilde{W}$ and $V \cap (D_1 \times \Delta^{'1})$ by the mappings $\psi_{\Tilde{W}}$ and $\psi_V(z^{m_W})$ from $D_1$. Since $\Phi_{\alpha} \circ \psi_{\Tilde{W}}$ is a parametrization of $W_1$, we obtain the following:

    \begin{align*}
        \| \Phi_{\alpha} \circ \psi_{\Tilde{W}}(z) - \psi_V (z^{m_W}) \| \leq \|  \Phi_{\alpha} \circ \psi_{\Tilde{W}} (z) - \psi_{\Tilde{W}} (z) \| + \|\psi_{\Tilde{W}}(z) - \psi_V (z^{m_W}) \|.
        \end{align*}
    Since the non-normal crossing type discriminant point $\alpha$ of $W$ is within a $\delta$-ball centered at $0 \in D_1$, the distance $\| \Phi_{\alpha} - id_{D_1} \|$ must be sufficiently small. Thus for sufficiently small $d_H(V,W)$ we can make the term $\| \Phi_\alpha \circ \psi_{\tilde W}(z) -  \psi_{\tilde W}(z)\| < \frac{\epsilon_1}{2}$ where the  $\epsilon_1$ associated with $V \cap (D \times \Delta'^1)$ was obtained from the Lyubich-Peters result.

    For any $W$ with sufficiently small $d_H(V, W)$, the \emph{vertical} distance $d_{sym}( \pi^{-1}_V (z), \pi^{-1}_W (z))$ of the sets $\pi^{-1}_V (z)$ and $\pi^{-1}_W (z)$ above $\partial D_1$ is defined as $$ d_{sym}( \pi^{-1}_V (z), \pi^{-1}_W (z)):= {\rm sup} \{{\rm inf} \{ d(z_k, w_j): j \in \{1, \dots, m_W\} \}: k \in \{1, \dots, m\} \},$$ where $\pi^{-1}_V (z):= \{z_1, \dots, z_{m}\}$ and $\pi^{-1}_W (z):= \{w_1, \dots, w_{m_W}\}$, and $d$ is the usual complex euclidean distance of $\mathbb{C}^2$. Now as $V$ is fixed, for any one-dimensional subvariety $W$ with its Hausdorff distance $d_H(V, W)$ bounded above by some sufficiently small upper bound, the vertical distance $d_{sym}(\pi^{-1}_V(z), \pi^{-1}_W (z))$ above $z \in \partial D_1$ must also be bounded above. Since $V$ is fixed, by shrinking $d_H(V, W)$, the quantity ${\rm sup}\{d_{sym}(\pi^{-1}_V(z), \pi^{-1}_W (z)): z \in \partial D_1\}$ can be made appropriately small. 

    Thus for sufficiently small $d_H(V, W)$, we can also make the term $\|\psi_{\Tilde{W}}(z) - \psi_V (z^{m_W}) \|$ smaller than $\frac{\epsilon_1}{2}$. Hence for any such $W$, we can obtain parametrizations $\Phi_{\alpha} \circ \psi_{\Tilde{W}}$ and $\psi_V$ such that $\|\Phi_{\alpha} \circ \psi_{\Tilde{W}}(z) - \psi_V (z^{m_W}) \| < \epsilon_1$.

    Applying the Lyubich-Peters result, we obtain $V \cap W \neq \emptyset$.
\end{proof}

\section{Applications of the sufficient condition}

\subsection{A holomorphic multifunction analog}

As a corollary, we shall prove an analog of the Lyubich-Peters result for holomorphic multifunctions. Recall that, the set of points in the image of a holomorphic multifunction $F: D \rightarrow (\mathbb{C}^{2})^{m}_{sym}$ inside $\mathbb{C}^{2}$ is denoted by $\{F(D)\}$: $$ \{F(D)\}:= \{f_j (z): F(z) = \langle f_1 (z), \dots, f_m (z) \rangle \in (\mathbb{C}^{2})^{m}_{sym}, z \in D \}.$$

\begin{corollaryC}\label{mfunction}

Suppose $F: D \rightarrow (\mathbb{C}^{2})^{m}_{sym}$ is a holomorphic multifunction, the image $\{F(D) \}$ is an irreducible subvariety of some open subset of $\mathbb{C}^{2}$, and $0 \in \{F(D)\}$ is a singular point of the image of the multifunction. Moreover suppose the polydisk $D_1 \times {\Delta}'^{2} \subset \mathbb{C} \times \mathbb{C}'^{2}$ is a good neighborhood of the multigraph of $F$ about $0$. Then there exists $\epsilon > 0$ such that for every holomorphic multifunction $G: D \rightarrow (\mathbb{C}^{2})^{m}_{sym}$ satisfying:
\begin{itemize}
    \item the multi-graph of $G$ has at most one non-normal crossing type discriminant point associated with the projection onto $D_1$,
    \item $\| F - G\|_{sym} < \epsilon$,
\end{itemize}
we have that $\{F(D)\} \cap \{G(D)\} \neq \emptyset$.

\end{corollaryC}

\begin{proof}
    Imitating the proof of Theorem \ref{main}, we obtain sufficiently close parametrizations of the images of $F$ and $G$ and show a non-empty intersection using the Lyubich-Peters result. Moreover, for holomorphic multifunctions $F,G : D \rightarrow (\mathbb{C}^{2})^{m}_{sym}$ with the symmetric product space distance $d_{sym} (F, G) < \epsilon$, the Hausdorff distance of the images of $F$ and $G$ must be less than $\epsilon$. Further, note that the maximal number of branches of the images of $F$ and $G$ are both $m$. Thus using the Puiseux technique as in the previous theorem, we can obtain suitably close parametrizations of $\{F(D)\}$ and $\{G(D)\}$.
\end{proof}

\begin{remark}
    Example $5.3$, below, shows that a holomorphic multifunction analog of the Lyubich-Peters result is not true without the hypothesis on the discriminant set. Note that the issue in the proof of Theorem \ref{main} about $W$ having potentially many more sheets than $V$ is not present in the Corollary \ref{mfunction}. 
\end{remark}

\subsection{The higher dimensional analog of Lyubich-Peters result}

The sufficient condition can be used to prove the higher dimensional analog of Lyubich-Peters's result for certain classes of holomorphic mappings $f: \Delta^{n} \rightarrow \mathbb{C}^{n+1}$. An appropriate higher dimensional analog of the Lyubich-Peters's result should involve mappings that preserve the dimensions of analytic objects, so, we naturally look to work with finite holomorphic mappings. Let us formulate a higher dimensional analog of Lyubich-Peters's result in the following:

\begin{question} \label{qhd}(\textbf{Higher dimensional analog of the result by Lyubich-Peters})
    \\
    Let $f: \Delta^{n} \rightarrow \mathbb{C}^{n+1}$ be a finite holomorphic mapping with $f(0)= 0$ as a singular point of $f(\Delta^{n})$. Does there exist an $\epsilon > 0$ such that for any finite holomorphic $g: \Delta^{n} \rightarrow \mathbb{C}^{n+1}$ with $\| f -g \| < \epsilon$, the intersection $f(\Delta^{n}) \cap g(\Delta^{n})$ is non-empty?
\end{question}

Using Theorem \ref{main}, we shall obtain a partial answer to the above question. We seek to formulate a higher-dimensional analog of Theorem \ref{main} below. To describe the higher dimensional setting, we start with an $n$-dimensional irreducible subvariety $V$ of some open neighborhood $\Omega_{n+1} \subset \mathbb{C}^{n+1}$ of the origin with $0 \in V^{\times}$, a neighborhood $D^{n} \times \Delta^{'1} \subset \subset \Omega_{n+1}$ around $0$ good for $V$. Suppose $\pi_{n+1} (0) \in D^{n}$ is a smooth point of the discriminant variety $D_V$. Further, notice that, after possibly a small rotation, a small enough polydisk $D^{(n-1)}_1 \times D_2$ inside $D^{n}$ good for $D_V$ around $\pi_{n+1}(0)$ can be chosen so that $D_V \cap (D^{(n-1)}_1 \times D_2) $ is a connected complex submanifold. With this setup, we obtain the following higher dimensional analog of Theorem \ref{main}:

\begin{cor}\label{hd}
  Suppose $V$, $D_V$, and $D^{n} \times \Delta^{'1}$ are as above. Then there exists $\epsilon > 0$ such that for all $n$~dimensional subvariety $W$ of $\Omega_{n+1}$ with $D_W \cap (D^{(n-1)}_1 \times D_2)$ being a connected complex submanifold and $d_{H}(V, W) < \epsilon$, the intersection $V \cap W$ is non-empty.
\end{cor}

\begin{proof}

Let $L_1 = D_V \cap (D^{(n-1)}_1 \times D_2)$ and $L'_1 = D_W \cap (D^{(n-1)}_1 \times D_2)$. Inside the polydisk $D^{(n-1)}_1 \times D_2$, we can describe the discriminant varieties $D_V$ and $D_W$ as multigraphs of holomorphic multifunctions from $D^{(n-1)}_1$ onto $D_2$. Since $L_1$ and $L'_1$ are both connected complex submanifolds inside $D^{(n-1)}_1 \times D_2$, by shrinking the good neighborhood if necessary, the inverse of the proper projection onto $D^{(n-1)}_1$ can be represented as a single sheeted holomorphic multifunction from $D^{(n-1)}_1$ to $D_2$.

 Now rotating the disk $D_2$ if necessary, we obtain a good direction $T_2$ and a good neighborhood $U_{n-1} \times T_2$ inside $D^{(n-1)}_1 \times D_2$ such that $(T_2 \times \Delta^{'1}) \cap V $ is a singular one-dimensional subvariety of $(T_2 \times \Delta^{'1})$. As $T_2$ is obtained by an arbitrarily small rotation of $D_2$, the discriminant varieties $(U_{n-1} \times T_2) \cap L_1$ and $(U_{n-1} \times T_2) \cap L'_1$ can again be represented as graphs of holomorphic functions from $U_{n-1}$ to $T_2$.

 Now for some $\alpha \in U_{n-1}$, we will have $(\{\alpha\} \times T_2) \cap L_1 = \pi_{n+1}(0)$. By the Lyubich-Peters result, there exists an $\epsilon > 0$ for the singular curve $L_1$ inside $(U_{n-1} \times T_2)$. As $U_{n-1} \times T_2$ is good for $L'_1$, the set $(\{\alpha\} \times T_2) \cap B_1$ is also singleton. This ensures that $(T_2 \times \Delta^{'1}) \cap W$ has one discriminant point inside $T_2$. Thus by Theorem \ref{main}, we have the result.
\end{proof}

As an immediate consequence of corollary \ref{hd}, we have the following result answering question \ref{qhd} partially. To describe the higher dimensional setting, we start with a finite holomorphic mapping $f: \Delta^{n} \rightarrow \mathbb{C}^{n+1}$ with $f(0)= 0$ as a singular point of $f(\Delta^{n})$, and a neighborhood $D^{n} \times \Delta^{'1}$ good for $f(\Delta^n)$ around $0$. Suppose $\pi_{n+1}:~f(\Delta^{n}) \cap (D^{n} \times \Delta^{'1}) \rightarrow D^{n}$ is the proper projection associated with the good neighborhood, $D_f$ is the discriminant variety associated with $\pi_{n+1}$, and $\pi_{n+1}(0)$ is a smooth point of $D_f$. Now $D^{n} \times \Delta^{'1}$ serves as a good neighborhood for any holomorphic image $g(\Delta^n)$ whenever the mapping $g:\Delta^n \rightarrow \mathbb{C}^{n+1}$ is sufficiently close to $f$, and thus the discriminant variety $D_g$ associated with the good neighborhood $D^{n} \times \Delta^{'1}$ can be obtained for each such finite holomorphic $g$. Further, notice that, after possibly a small rotation, a small enough polydisk $D^{(n-1)}_1 \times D_2$ \emph{good} for $D_f$ around $\pi_{n+1}(0)$ can be chosen so that $D_f \cap (D^{(n-1)}_1 \times D_2)$ is a connected complex submanifold. We obtain the following special case of the higher dimensional analog of Lyubich-Peters' result:

\begin{corollaryB}\label{hd1}
    Let $f$, $D_f$ and $D^{(n-1)}_1 \times D_2$ are as above. Then there exists $\epsilon > 0$ such that for all finite holomorphic $g: \Delta^{n} \rightarrow \mathbb{C}^{n+1}$ with $D_g \cap (D^{(n-1)}_1 \times D_2)$ being a connected complex submanifold and $\| f -g \| < \epsilon$, the intersection $f(\Delta^{n}) \cap g(\Delta^{n})$ is non-empty.
\end{corollaryB}

\begin{proof}
    Since $\| f - g\| < \epsilon$ forces the Hausdorff distance $d_{H}(f(\Delta^{n}), g(\Delta^{n}))$ of $f(\Delta^{n})$ and $g(\Delta^{n})$ to be smaller than $\epsilon$, the above is obtained from a direct application of Corollary \ref{hd}. 
\end{proof}

\section{Examples}

In this section, we present four examples to illustrate various points. We explicitly prove the Theorem \ref{main} for two particular varieties in the Example $5.1$. Example $5.2$ shows that the sufficient condition on Theorem \ref{main} is not a \emph{necessary} \emph{condition}. In Example $5.3$ we illustrate a case where the sufficient condition is required for the intersection to be non-empty. We explicitly prove the Corollary \ref{hd1} for particular holomorphic functions in the last example. 

\begin{example}

Consider the analytic varieties $$V:= \{ (z,w) \in \mathbb{C}^{2}: z^{2} - w^{3} = 0 \},$$  $$V_{\epsilon}:= \{ (z,w) \in \mathbb{C}^{2} : (z - \epsilon_{1})^{2} - (w - \epsilon_{2})^{3} = 0 \} ,$$ where $\epsilon_j \neq  0$ for $j= 1, 2$. Note that $V_{\epsilon}$ is just $V$ shifted by $\epsilon = (\epsilon_1, \epsilon_2)$.

We can describe both $V$ and $V_{\epsilon}$ as the multigraph of the holomorphic multifunctions $F, F_{\epsilon}: D \longrightarrow (\mathbb{C}^{'1})^{3}_{sym}$ defined as $$\alpha \mapsto \left \langle \alpha^{\frac{2}{3}}, \alpha^{\frac{2}{3}} \omega, \alpha^{\frac{2}{3}} \omega^{2} \right \rangle, $$ $$\alpha \mapsto \left \langle (\alpha - \epsilon_1)^{\frac{2}{3}} + \epsilon_2, \omega (\alpha - \epsilon_1)^{\frac{2}{3}} + \epsilon_2, \omega^{2} (\alpha - \epsilon_1)^{\frac{2}{3}} + \epsilon_2 \right \rangle, $$ where $\omega$ is the cube-root of unity. 

The discriminant sets $D(V)$ and $D(V_{\epsilon})$ of $V$ and $V_{\epsilon}$ are $\{0\}$ and $\{ \epsilon_1 \}$ respectively. But $V$ has a local parametrization $t \mapsto (t^{3}, t^{2})$. The function $f(t)= (t^{3} - \epsilon_{1})^{2} - (t^{2} - \epsilon_{2})^{3}$ is a degree four polynomial on $t$. 

Since $f(t)= (t^{3} - \epsilon_{1})^{2} - (t^{2} - \epsilon_{2})^{3} = \epsilon^2_1 + \epsilon^3_2 -2 t^3 \epsilon_1 + 3t^4 \epsilon_2 - 3 t^2 \epsilon^2_2$, and the function $-2 t^3 \epsilon_1 + 3t^4 \epsilon_2 - 3 t^2 \epsilon^2_2$ has a zero at $t=0$, we can prove $f(t)$ has a zero using the Rouche's theorem. 

In the case when $|\frac{\epsilon_2}{\epsilon_1}| \geq 1$ we have: $$ |\epsilon_2| \left | 3 t^4 - 3t^2 \epsilon_2 - 2t^3 \frac{\epsilon_1}{\epsilon_2} \right | \geq |\epsilon_2| \left ( 3 - 3 |\epsilon_2| - 2 |\frac{\epsilon_1}{\epsilon_2}| \right ),$$ on the boundary of the unit disk $t \in \delta D$. For $\| \epsilon\|$ sufficiently small, the term $3 - 3 |\epsilon_2| - 2 |\frac{\epsilon_1}{\epsilon_2}|$ is close to $1$ or is bigger than $\frac{1}{2}$, and the term $\epsilon^2_2 + \frac{\epsilon^2_1}{\epsilon_2}$ is close to $0$, providing us the following strict inequality on the boundary $\partial D$, $$ |\epsilon_2| \left ( 3 - 3 |\epsilon_2| - 2 |\frac{\epsilon_1}{\epsilon_2}| \right ) > |\epsilon_2| \left |\epsilon^2_2 + \frac{\epsilon^2_1}{\epsilon_2} \right |.$$ Again Rouche's theorem forces $f(t)$ to have zeros for sufficiently small $\| \epsilon\|$.

In the case when $|\frac{\epsilon_1}{\epsilon_2}| > 1$ we have: $$ |\epsilon_2| \left | 2t^3 \frac{\epsilon_1}{\epsilon_2} + 3t^2 \epsilon_2 - 3t^4 \right | \geq |t|^2 |\epsilon_2| \left (2 |t| |\frac{\epsilon_1}{\epsilon_2}| - 3 |\epsilon_2| - 3|t|^2 \right ).$$ There exists $0 < \delta < 1$ small enough such that on the circle $|t| = \delta$, the term $2 |t \frac{\epsilon_1}{\epsilon_2}|~-~3|\epsilon_2|-~3|t|^2$ is always positive for sufficiently small $\| \epsilon \|$. Clearly for small $\|\epsilon \|$, the term $2 |t| |\frac{\epsilon_1}{\epsilon_2}| - 3 |\epsilon_2| - 3|t|^2$ is close to some positive number bigger or equal to $2 \delta - 3 \delta^2 > 0$ on the circle $|t| = \delta$, and $\epsilon^2_2 + \frac{\epsilon^2_1}{\epsilon_2}$ is close to $0$. Thus on the circle $|t| = \delta$, we obtain $$ |t|^2 |\epsilon_2| \left (2 |t| |\frac{\epsilon_1}{\epsilon_2}| - 3 |\epsilon_2| - 3|t|^2 \right ) > |\epsilon_2| \left |\epsilon^2_2 + \frac{\epsilon^2_1}{\epsilon_2} \right |.$$ Thus Rouche's theorem implies $f(t)$ admits zeros inside $|t| < \delta$.

Thus $V \cap V_{\epsilon}$ is non-empty whenever $\epsilon_1, \epsilon_2$ is sufficiently small.

\end{example}

\begin{example}

Consider the analytic varieties $$V:= \{ (z,w) \in \mathbb{C}^{2}: z^{2} - w^{3} = 0 \},$$  $$V_{\epsilon}:= \{ (z,w) \in \mathbb{C}^{2} : (z - \epsilon_{1})^{2} - (w - \epsilon_{2})^{3} = \epsilon^{2}_1 \epsilon^{3}_2 \} ,$$ where $\epsilon_j \neq  0$ for $j=1,  2$. Let, $\epsilon_3 = \epsilon^{2}_1 \epsilon^{3}_2$.

$V_{\epsilon}$ is the multigraph of the holomorphic multifunction $F_{\epsilon}: D \longrightarrow (\mathbb{C}^{'1})^{3}_{sym}$ defined as $$\alpha \mapsto \left \langle ((\alpha - \epsilon_1)^{2} - \epsilon_{3})^{\frac{1}{3}} + \epsilon_2, \omega ((\alpha - \epsilon_1)^{2} - \epsilon_{3})^{\frac{1}{3}} + \epsilon_2, \omega^{2} ((\alpha - \epsilon_1)^{2} - \epsilon_{3})^{\frac{1}{3}} + \epsilon_2  \right \rangle. $$ 

The discriminant set $D(V_{\epsilon})$ of $V_{\epsilon}$ is $\{ \epsilon_1 - \sqrt{\epsilon_3}, \epsilon_1 + \sqrt{\epsilon_3} \}$.

But $V$ has a local parametrization $t \mapsto (t^{3}, t^{2})$. The equation $(t^{3} - \epsilon_{1})^{2} - (t^{2} - \epsilon_{2})^{3} = \epsilon_{3}$ is a degree four polynomial on $t$. Using the techniques of the previous example we can study separate cases of $\epsilon_j$'s and apply Rouche's theorem to establish that $V \cap V_{\epsilon}$ is non-empty in each case. Thus although $V_{\epsilon}$ does not meet the requirement of the main result, $V \cap V_{\epsilon}$ is non-empty, showing that the condition in the main result is not a necessary condition.
\end{example}

\begin{example}
   
Consider the analytic varieties $$V:= \{ (z,w): z^{2}-w^{3}=0 \},$$ $$V_{\epsilon}:= \{ (z,w): z^{2}-w^{3}= \epsilon \},$$ where $\epsilon \neq 0$. 

We can describe $V$ as the multigraph of the holomorphic multifunction $F: D \rightarrow~(\mathbb{C}^{'1})^{3}_{sym}$ defined similar to Example $5.1$. Similarly $V_{\epsilon}$ can be expressed as the multigraph of $F_{\epsilon}$, defined as $$\alpha \mapsto \left \langle (\alpha^{2} - \epsilon)^{\frac{1}{3}}, (\alpha^{2} - \epsilon)^{\frac{1}{3}} \omega, (\alpha^{2} - \epsilon)^{\frac{1}{3}} \omega^{2} \right \rangle, $$ where $\omega$ denotes the cube root of unity. The discriminant set corresponding to the multigraph of $F$ is $\{0 \} \subset D$ and for the multigraph of $F_{\epsilon}$ the discriminant set is $\{ \sqrt{\epsilon}, - \sqrt{\epsilon} \}$. Thus for any $\epsilon > 0$ the discriminant set corresponding to multigraph of $F_{\epsilon}$ always has two points and meanwhile, $V \cap V_{\epsilon} = \emptyset$, showing that our hypothesis on the discriminant set was needed.
\end{example}

\begin{example}(\textbf{An example of the higher dimensional analog result})

Consider the mapping from the unit polydisk in $\mathbb{C}^{2}$ onto a neighborhood of the origin in $\mathbb{C}^{3}$ given by, $$ f: (z,w) \mapsto (z^{2}, w^{3} + z^{3}, w). $$ The image of $f$ can be shown to lie inside the hypervariety $V$ in $\mathbb{C}^{3}$ containing the origin as a singular point which is given by the following equation, $$ a^{3}= (b - c^3)^{2},$$ where $(a,b,c)$ are the coordinates in $\mathbb{C}^{3}$. As $V$ is irreducible, the image of $f$ coincides with it. Since the projection $\pi_2$ to the last two coordinates is proper, we can describe $V$ locally near $0 \in \mathbb{C}^{3}$ via the multifunction $$ \pi^{-1}_2 : (b,c) \mapsto \left \langle ( (b - c^{3})^{\frac{2}{3}}, b, c), ( \omega (b - c^{3})^{\frac{2}{3}}, b, c), ( \omega^{2} (b - c^{3})^{\frac{2}{3}}, b, c) \right \rangle, $$ where $\omega$ stands for the cube root of unity. So the discriminant variety is given by $D_V := \{(b,c): (b - c^{3})^{2} = 0\}$.

Note that as $V$ is an irreducible hypervariety, the singular set of $V$ can be written as $\{ (a,b,c) \in V: 3a^{2} = 2(b -c^{3}) = -6c^{2}(b - c^{3}) = 0\}$. In particular $\{(0, \delta^{3}, \delta): 0 < | \delta | < 1 \}$ is a subset of the singular set of $V$.

Consider the following perturbation of $f$ from the unit polydisk $D^{2}$ in $\mathbb{C}^{2}$ onto a neighborhood of the origin in $\mathbb{C}^{3}$ given by $$ g: (z,w) \mapsto ( z^2 + \epsilon_1 w^4, w^{3} + z^{3} + \epsilon_{2} w^{2}, w + \epsilon_3 ),$$ where $\epsilon_j \neq 0$ are small for $j= 1,2,3$. The image of $g$ lies inside a hypervariety $W$ of $\mathbb{C}^{3}$ which is given by the following equation, $$ (a - \epsilon_1(c - \epsilon_3)^{4})^{3} = (b - (c - \epsilon_{3})^{3} - \epsilon_2 (c - \epsilon_3)^{2} )^{2}.$$ Again as the projection $\pi_2$ in the last two coordinates is proper, the following holomorphic multifunction provides a local description of $W$ in a neighborhood of the origin, $$ (b,c) \mapsto \left \langle ( B^{\frac{2}{3}} + \epsilon_1(c - \epsilon_3)^{4}, b, c), (\omega B^{\frac{2}{3}} + \epsilon_1(c - \epsilon_3)^{4}, b, c), (\omega^{2} B^{\frac{2}{3}} + \epsilon_1(c - \epsilon_3)^{4}, b, c) \right \rangle, $$ where $B = (b - (c - \epsilon_{3})^{3} - \epsilon_2 (c - \epsilon_3)^{2} )$, and $\omega$ stands for the cube root of unity. So the discriminant variety is given by $D_W := \{(b,c): B^{2} = 0\}$.

Now both $D_V$ and $D_W$ can be written as holomorphic multifunctions $$b \mapsto \langle (b, (b)^{\frac{1}{3}}): 2, (b, \omega (b)^{\frac{1}{3}}): 2, (b, \omega^{2} (b)^{\frac{1}{3}}):2 \rangle ,$$ $$ b \mapsto \langle (b, c_1(b)): 2, (b, c_2(b)): 2, (b, c_3(b)): 2 \rangle,$$ where $c_j(b)$ are the zeros of $b - (c - \epsilon_{3})^{3} - \epsilon_2 (c - \epsilon_3)^{2}$. Note that for appropriate non-zero values of $b$, both $D_V$ and $D_W$ generically have three distinct branches, each of multiplicity~$2$.

For a nonzero $\delta$, the singular point $(0, \delta^{3}, \delta)$ of $V$ projects down to a non-singular point of $D_V$. By making a careful choice of $\delta$, we can obtain that for sufficiently small $\epsilon_j$s, there is exactly one multiplicity $2$ branch of $D_W$ in a small neighborhood of $( \delta^{3}, \delta)$. Thus applying Corollary \ref{hd}, we can say $f(D^{2}) \cap g(D^{2}) \neq \emptyset $.

On the other hand, to directly see that the intersection is non-empty, we plug in the finite mapping $g$ in the equation of $V$. We obtain the following expression, $$ p(z,w) = (z^{2} + \epsilon_1 w^{4})^{3} - (w^3 + z^3 + \epsilon_2 w^2 - (w+ \epsilon_3)^3)^{2}.$$  Now if there exists $(z, w) \in D^{2}$ for which $g$ satisfies the equation of $V$, i.e. $p(z,w)$ is zero then the point $g(z, w)$ must be lying inside $V \cap W$. Setting $w = 0$, and $z = \frac{\epsilon_3}{(2)^{\frac{1}{3}}}$, we obtain a zero of $p(z, w)$. So $V \cap W = f(D^2) \cap g(D^2)$ is non-empty.

\end{example}

 

\end{document}